\documentclass[11pt]{amsart}
 
\usepackage[margin=1in]{geometry} 
\usepackage{amsmath,amsthm,amssymb,graphicx,tikz-cd,hyperref,mathtools}
\usepackage[shortlabels]{enumitem}
\usepackage[linesnumbered,ruled,noend]{algorithm2e}
\usepackage[style=alphabetic,maxbibnames=99]{biblatex}

\renewbibmacro{in:}{}
\DeclareFieldFormat*{title}{#1}
\hypersetup{colorlinks,linkcolor=blue,citecolor=blue,urlcolor=blue}

\newcommand{\commentr}[1]{\tcp*[r]{\parbox[t]{7.2cm}{\raggedright\normalfont{#1}}}}
\newcommand{\commentf}[1]{\tcp*[f]{\parbox[t]{7.2cm}{\raggedright\normalfont{#1}}}}

\SetKwComment{tcp}{}{}

\definecolor{darkgreen}{rgb}{0.06, 0.5, 0.06}

\newcommand{\RS}[1]{}
\newcommand{\KL}[1]{}
\newcommand{\KS}[1]{}
\newcommand{\HT}[1]{}
\newcommand{\ML}[1]{\textcolor{blue}{{\sf (Matt:} {\sl{#1})}}}

\newcommand{\SA}[1]{}

\newcommand{\N}{\mathbb{N}}
\newcommand{\F}{\mathbb{F}}
\newcommand{\Z}{\mathbb{Z}}
\newcommand{\R}{\mathbb{R}}

\newcommand{\Dcal}{\mathcal{D}}
\newcommand{\Mcal}{\mathcal{M}}
\newcommand{\Gcal}{\mathcal{G}}
\newcommand{\Ccal}{\mathcal{C}}

\newcommand{\Max}{\mathrm{Max}}

\newcommand{\OP}[1]{\left( #1 \right)}

\newtheorem{innercustomthm}{Theorem}
\newenvironment{customthm}[1]
  {\renewcommand\theinnercustomthm{#1}\innercustomthm}
  {\endinnercustomthm}

\newtheorem{theorem}{Theorem}[section]
\newtheorem{proposition}[theorem]{Proposition}
\newtheorem{corollary}[theorem]{Corollary}
\newtheorem{lemma}[theorem]{Lemma}

\theoremstyle{definition}
\newtheorem{notation}[theorem]{Notation}
\newtheorem{definition}[theorem]{Definition}
\newtheorem{remark}[theorem]{Remark}

\begin{document}
 
\title{Connectivity of Markoff mod-p graphs and maximal divisors}
\author[J.~Eddy, E.~Fuchs, M.~Litman, D.~Martin, N.~Tripeny]{Jillian Eddy, Elena Fuchs, Matthew Litman, Daniel Martin, and Nico Tripeny}

\maketitle

\begin{abstract}
     Markoff mod-$p$ graphs are conjectured to be connected for all primes $p$. In this paper, we use results of Chen and Bourgain, Gamburd, and Sarnak to confirm the conjecture for all $p > 3.448\cdot10^{392}$. We also provide a method that quickly verifies connectivity for many primes below this bound. In our study of Markoff mod-$p$ graphs we introduce the notion of \emph{maximal divisors} of a number. We prove sharp asymptotic and explicit upper bounds on the number of maximal divisors, which ultimately improves the Markoff graph $p$-bound by roughly 140 orders of magnitude as compared with an approach using all divisors. 
\end{abstract}

\section{Introduction} 
The \textit{Markoff equation} is given by 
    \begin{equation}\label{Markoffequation}
        x^2+y^2+z^2=xyz,
    \end{equation}
and non-negative integer solutions $(a,b,c)$ to this equation are called \textit{Markoff triples}. An integer that is a member of such a triple is called a \textit{Markoff number}. Since their introduction by Andrey Markoff in \cite{Markoff}, Markoff triples have arisen in many different contexts across the mathematical landscape. Recently, Bourgain-Gamburd-Sarnak have explored various arithmetic properties of Markoff triples (see \cite{BGSannouncement}), proving that there are infinitely many composite Markoff numbers. A key ingredient in the proof of this fact is a combinatorial property that we describe below.

Markoff triples can be realized as vertices of a \emph{Markoff tree} as follows (note that Markoff triples with negative entries can be realized in a nearly identical way, but we focus on the positive triples here for ease of exposition). Let $R_1, R_2,$ and $R_3$ be involutions acting on triples of numbers defined by
\begin{equation}\label{involutions}
R_1(a,b,c)=(bc-a,b,c),\; R_2(a,b,c)=(a,ac-b,c),\; R_3(a,b,c)=(a,b,ab-c)
\end{equation}
and note that each of these involutions sends a Markoff triple to another Markoff triple. In fact, all positive Markoff triples can be realized as some word in these involutions applied to the triple $(3,3,3)$.

\begin{figure}[h!] 
    \label{fig1}
    \centering
    \includegraphics[scale=.45]{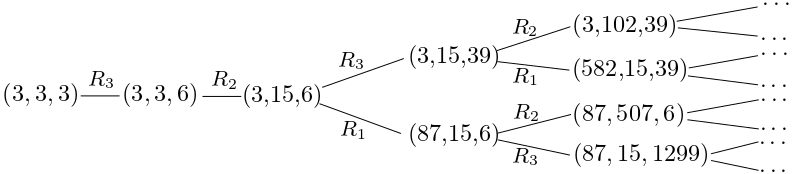}
    \caption{A branch of the Markoff tree generated by applying the involutions $R_1,R_2,R_3$ to the fundamental solution (3,3,3). \iffalse The tree shows one of several very similar branches of the tree, the others are generated by acting on $(1,1,1)$ via $R_1$ and $R_2$, as well as letting one other involution act on the triple immediately adjacent to $(1,1,1)$ (that other involution would be $R_1$ acting on $(1,1,2)$). Those branches will simply contain permutations of the triples shown in this tree.\fi}
    \label{Markofftree}
\end{figure}

In studying the arithmetic of Markoff numbers, it is natural to consider the solutions to (\ref{Markoffequation}) mod $p$: understanding this set is crucial to sieving on the set of Markoff numbers and is behind Bourgain-Gamburd-Sarnak's result on composite Markoff numbers. More specifically, it is useful to consider a version of the Markoff tree described above modulo primes $p$. These graphs $\Gcal_p$, which we call \emph{Markoff mod-$p$ graphs}, are constructed as follows. The vertex set of this graph is the set of nonzero solutions to (\ref{Markoffequation}) mod $p$, and two vertices $v_1,v_2$ are connected by an edge if 
$$R_i(v_1)\equiv v_2\pmod{p} \mbox{ for some } 1\leq i\leq 3.$$

Baragar was the first to conjecture that this graph is connected for any prime \cite{B}. A deep result of Bourgain-Gamburd-Sarnak in \cite{BGS} has confirmed this for all primes outside a density zero subset. Specifically, they show the following.

\begin{theorem}[Theorems 1, 2 Bourgain-Gamburd-Sarnak \cite{BGS}]\label{sizecomponent} Fix $\varepsilon>0$. Then for sufficiently large $p$ there is a connected component $\Ccal_p$ of $\Gcal_p$ for which
$$|\Gcal_p \backslash \Ccal_p| \leq p^\varepsilon$$
(note that $|\Gcal_p|\sim p^2$), and any connected component $\Ccal$ of $\Gcal_p$ satisfies
$|\Ccal|\gg (\log p)^{1/3}$. Moreover, for $\varepsilon' > 0$ and sufficiently large $t$, the number of primes $p \leq t$ for which $\Gcal_p$ is not connected is at most $t^{\varepsilon'}$.
\end{theorem}

The bound on $|\Gcal_p\backslash\Ccal_p|$ was thereafter made much more explicit in \cite{KMSV}, where it was shown that the exponent of $1/3$ in the bound on $|\Ccal|$ can be improved to $7/9$. Bourgain-Gamburd-Sarnak conjecture that these graphs make up an expander family, and this has been further explored in \cite{dCL} and \cite{dCM} (from which it appears that certain subfamilies of this family are actually Ramanujan).

Subsequently, Chen \cite{chen} proved that the size of any connected component of $\Gcal_p$ must be divisible by $p$. This implies that if $\Gcal_p$ is disconnected, meaning $\Gcal_p\backslash\Ccal_p$ is a nonempty union of connected components, then $|\Gcal_p\backslash\Ccal_p|\geq p$. So by making explicit the phrase ``sufficiently large" in Theorem \ref{sizecomponent}, particularly for $\varepsilon=1$, we obtain a lower bound on primes $p$ for which $\Gcal_p$ is necessarily connected. 

In Section \ref{firstbound} we refine the arguments in \cite{BGS} and make their asymptotic bounds explicit. The result combines with Chen's theorem to prove that $\Gcal_p$ is connected for $p > 10^{532}$ (Corollary \ref{cor:prelim}). 

Section \ref{MaxDivsec} introduces maximal divisors, the main tool behind further reduction to our $p$-bound. 

\begin{definition}
Let $n$ be a positive integer, and let $x\in\R$. A positive divisor $d$ of $n$ is \emph{maximal with respect to $x$} if $d\leq x$ and there is no other positive divisor $d'$ of $n$ such that $d'\leq x$ and $d\,|\,d'$.  The set of maximal divisors with respect to $x$ is denoted $\Mcal_x(n)$.\end{definition}

In other words, a maximal divisor is a maximal element in the partially ordered (by divisibility) set of divisors of $n$ that are less than $x$.

This definition is motivated by a task that appears often in \cite{BGS}: to bound a sum over the union of subgroups of order at most $x$ in the cyclic group of order $n$. Since a group element may belong to many such subgroups, overcounting is avoided by rewriting the sum using inclusion-exclusion, and the very first term of the result is a sum over maximal divisors. (Details are in the next section.) 

Our approach in Section \ref{MaxDivsec} is designed to give explicit bounds on $|\Mcal_x(n)|$ for any $x$ and for computationally-feasible sized $n$---up to $10^{532}$ as dictated by Corollary \ref{cor:prelim}. But our approach also happens to furnish a simple proof of a sharp asymptotic bound.

\begin{theorem}\label{thm:intro}For any $\varepsilon > 0$, if $\alpha\in[\varepsilon,1-\varepsilon]$ then $$\log|\Mcal_{n^\alpha}(n)|=\log\!\left(\frac{1}{\alpha^\alpha(1-\alpha)^{1-\alpha}}\right)\!\frac{\log n}{\log\log n}+O\!\left(\frac{\log n}{(\log\log n)^2}\right).$$ The implied constant depends only on $\varepsilon$.\end{theorem}

As an immediate corollary, we also obtain a similar bound on the total number of divisors of $n$ less than $x$ (Corollary \ref{cor:asym}). These results can be viewed as generalizations of Wigert's theorem: $\log\tau(n) = (\log 2+o(1))\log n/\log\log n$, where $\tau(n)$ is the number of divisors of $n$ \cite{Wigert}. (The constant $\log 2$ is recovered by setting $\alpha=1/2$ in Theorem \ref{thm:intro}.)

In Section \ref{prooftheorem} we use our work on maximal divisors to prove our main result.

\begin{theorem}\label{thm:main}$\Gcal_p$ is connected for all primes $p > 863\#53\#13\#7\#5\#3^32^5\approx 3.448\cdot10^{392}$, where $n\#$ denotes the product of primes less than or equal to $n$.\end{theorem}

The lower bound in Theorem \ref{thm:main} was output by a computer using Algorithm \ref{alg}, which determines the exact point at which our method for proving connectivity via maximal divisors fails.

Finally, in Section~\ref{data} we provide data on the proportion of smaller primes for which we can also verify connectivity of $\Gcal_p$. As Table~\ref{tab:percentage} shows, our approach begins to work for a significant proportion of primes at around $10^8$, and for $22\leq n\leq 90$ it proves connectivity for 10,000 out of 10,000 randomly chosen primes between $10^n$ and $10^{n+1}$. Note that there are still primes for which our connectivity check fails up until the bound from Theorem \ref{thm:main}. Table~\ref{tab:percentage}'s success for smaller primes is due to the expected number of divisors of $p\pm 1$ being much less than the maximum possible number of divisors. This ability to check for connectivity for smaller primes would be useful, for example, in a recent application of Markoff triples to a cryptographic hash function in \cite{FLLT}, in which one needs to be able to check connectivity of a Markoff mod-p graph for a specific large (but still manageable using our criterion) prime $p$ in order to construct the hash.

Interestingly, our data reveals that already for primes of size $10^{31}$, the Erd\"{o}s-Kac theorem takes over in the sense that the expected value of $\tau(p\pm1)$ is small enough so that it becomes extremely rare to need the improvement that comes by considering maximal divisors rather than all divisors. This is one hint that  our methods via maximal divisors alone will not prove connectivity of all Markoff graphs, and that this will require new insight.

\vspace{0.1in}

\noindent{\bf Acknowledgements:} This project was started at the UC Davis 2021 REU, and we thank Javier Arsuaga and Greg Kuperberg for the REU's creation and organization. We also thank Matthew de Courcy-Ireland for helpful conversations and comments on this work.


\section{A preliminary bound}\label{firstbound}

In this section, we prove a preliminary bound towards Theorem~\ref{thm:main}, which will not only serve to introduce the reader to the key points of our main argument, but will also be necessary in the proof of Theorem~\ref{thm:main}. The Appendix, which serves to make several statements in \cite{BGS} more precise, will feed into the technical details of the proofs.

We use the following parameterization, which matches that of Bourgain, Gamburd, and Sarnak up to a change of variables (equations (15), (16), and (18) in \cite{BGS}). A triple $(a,b,c)\in\F_p$ with $a\neq 0,\pm2$ solves $x^2+y^2+z^2=xyz$ if and only if it is of the form \begin{equation}\label{eq:7}\left(r+r^{-1},\,\frac{(r+r^{-1})(s+s^{-1})}{r-r^{-1}},\,\frac{(r+r^{-1})(rs+r^{-1}s^{-1})}{r-r^{-1}}\right)\end{equation} for some $r,s\in\F_{p^2}$. The orbit of this triple under the Vieta involutions that fix the first coordinate, called $R_2$ and $R_3$ in (\ref{involutions}), consists precisely of triples of the form \begin{equation}\label{eq:rotorbit}\left(r+r^{-1},\,\frac{(r+r^{-1})(r^{2n}s+r^{-2n}s^{-1})}{r-r^{-1}},\,\frac{(r+r^{-1})(r^{2n\pm1}s+r^{2n\pm1}s^{-1})}{r-r^{-1}}\right)
\end{equation}
for some $n\in\Z$, and one can similarly describe the orbits that fix the second or third coordinate, as well. So the number of triples in this orbit depends on the multiplicative order of $r$ in $\F_{p^2}^*$.

Note that in \cite{BGS}, connectivity is proven for a slightly modified Markoff mod-$p$ graph, where the edges are defined not by the involutions $R_i$ as above, but by so-called rotations that they denote $\textrm{rot}(x_k)$, but this is in essence the same as product $R_iR_j$ where $\{i,j,k\}=\{1,2,3\}$.

Our strategy, based off of \cite{BGS}, is to assign an order to every triple in $\Gcal_p$ as follows. Given $a=r+r^{-1}$ as above, let $\textrm{ord}_p(a)$ be the multiplicative order of $r$ in $\F_{p^2}^*$. This agrees with the notion of order in \cite{BGS} (see their equations (8) and (9)) unless $a=\pm 2$, but it is shown in \cite{BGS} that a triple with $\pm 2$ in some coordinate is necessarily in the large connected component, so we need not consider this case for our purposes. Define the order of $(a,b,c)$ to be
\begin{equation}\label{order}
\textrm{Ord}_p((a,b,c)):=\max\{\text{ord}_p(a),\text{ord}_p(b),\text{ord}_p(c)\}\end{equation}

One of the key ideas in Bourgain-Gamburd-Sarnak's proof of the connectivity of $\Gcal_p$ is that, if a triple $(a,b,c)\in \Gcal_p$ has large enough order in the above sense, then there is always a triple of larger order in one of the orbits of 
$\langle R_i,R_j\rangle$ acting on $(a,b,c)$. One then walks along these orbits in what Bourgain-Gamburd-Sarnak call the Middle Game of the proof, increasing the order gradually, until one gets to a triple of order roughly $p^{1/2}$ (see Proposition~\ref{preciseBGSEG} in our Appendix for a precise statement), which is then necessarily connected to the large connected component $\mathcal C_p$ in Theorem~\ref{sizecomponent}. So, all triples of large enough order are connected to each other, and the question is then, how many triples potentially do not have large enough order, and hence may not be in $\mathcal C_p$? According to Chen \cite{chen}, the number of these bad triples not connected to $\mathcal C_p$ must be divisible by $p$. Hence, if we can show that this number is strictly less than $p$, we may deduce that there are no bad triples at all and, in fact, $\Gcal_p$ is connected. In fact, we can loosen this a bit as we explain in Lemma~\ref{lem:divisible} below.

We recall that a central ingredient in the Middle Game of \cite{BGS} is an upper bound on the number of triples of order at most $t$ in the orbit (\ref{eq:rotorbit}) and its analogues in which coordinates other than the first one are fixed. Without loss of generality, assume this maximal coordinate is the first one. Using the parametrization in (\ref{eq:rotorbit}), we have the following lemma, which sharpens the bound used by Bourgain-Gamburd Sarnak at the start of Section 4 in \cite{BGS} when they reference a bound by Corvaja-Zannier in \cite{corvaja}.

\begin{lemma}\label{prop:corvaja}If $r\in\F_{p^2}^*$ has order $t>2$, then the number of congruence classes $n\pmod{t}$ for which \emph{ord}$_p((r+r^{-1})(sr^n+(sr^n)^{-1})/(r-r^{-1}))$ divides $d$ is at most $\tfrac{3}{2}\max((6td)^{1/3},4td/p)$.
\end{lemma}

\begin{proof}The number of congruence classes in question is bounded by half the number of solutions $(x,y)\in \overline{\F_p}^2$ to the system of equations $x^{t}=1$, $y^{d}=1$, and $$\frac{(r+r^{-1})(sx+(sy)^{-1})}{r-r^{-1}}=y+y^{-1}.$$ (We halve the number of solutions because $(x,y)$ and $(x,y^{-1})$ only give one congruence class, yet get counted as distinct solutions unless $y=\pm1$. But as mentioned in the introduction, the case $y=\pm 1$ is ignored as any triple with coordinate $\pm2$ is known to be in $\Ccal_p$.) Solutions to the last equation above lie on the projective curve $C$ defined by 
\begin{equation}\label{czcurve}
\frac{s(r+r^{-1})}{r-r^{-1}}X^2Y-XY^2-XZ^2+\frac{r+r^{-1}}{s(r-r^{-1})}YZ^2=0.
\end{equation}
Assume $r+r^{-1}\neq 0$ since otherwise the proposition is trivial to check (and not useful). Along with $r+r^{-1}\neq\pm (r-r^{-1})$, which is always true, this implies $C$ is smooth. Therefore we can apply Theorem 2 in \cite{corvaja} to the rational functions $u([X,Y,Z])=(X/Z)^{t}$ and $v([X,Y,Z])= (Y/Z)^{d}$. The zeros and poles of $u$ or $v$ that lie on $C$ are $[1,0,0]$, $[0,1,0]$, and $[0,0,1]$. The Euler characteristic of $C\backslash\{[1,0,0],[0,1,0],[0,0,1]\}$ as defined in \cite{corvaja} is $$\chi=\big|\{[1,0,0],[0,1,0],[0,0,1]\}\big|+2\binom{\deg C-1}{2}-2 = 3.$$ By \cite{corvaja}, the number of points on $C$ that solve $u([X,Y,Z])=v([X,Y,Z])=1$ is bounded from above by $3\max((2\chi\deg u\deg v)^{1/3},4\deg u\deg v/p)$. The claim follows.\end{proof}

In Section~\ref{firstbound}, we mentioned Chen's result from \cite{chen} that any connected component in $\mathcal G_p$ has size divisible by $p$. We combine this with a few observations about the Markoff graphs to yield the following.

\begin{lemma}\label{lem:divisible}If $p>3$, then the number of vertices in $\Gcal_p\backslash \mathcal C_p$ is divisible by $4p$.
\end{lemma}

\begin{proof}Chen proved that the number of vertices in any connected component of $\Gcal_p$ is divisible by $p$ \cite{chen}. To prove divisibility by $4$, it suffices to show that $\Gcal_p\backslash \mathcal C_p$ is closed under negating any pair of coordinates. Indeed, no triple has a $0$ in two coordinates, so $(a,b,c)$, $(a,-b,-c)$, $(-a,b,-c)$, and $(-a,-b,c)$ are always distinct. 

If $p\equiv 1\,\text{mod}\,4$, then negating any two coordinates of a triple of order $p-1$ also has order $p-1$. If $p\equiv 3\,\text{mod}\,4$, then negating any two coordinates of a triple of order $p+1$ also has order $p+1$. In particular, we can always find some $(a_0,b_0,c_0)\in \mathcal C_p$ such that $(a_0,-b_0,-c_0)$, $(-a_0,b_0,-c_0)$, and $(-a_0,-b_0,c_0)$ are also in $\mathcal C_p$. Since negating any two coordinates in a pair of path-connected triples leaves them path-connected, we see that $\mathcal C_p$ is closed under negating of any pair of coordinates. This implies the same is true of $\Gcal_p\backslash \mathcal C_p$.\end{proof}

\begin{remark}
The $4p$ in Lemma~\ref{lem:divisible} could be improved to $12p$ by proving that $(3,3,3)\in\Ccal_p$. According to \cite{BGS}, this would be true if $(3,3,3)$ is connected to a triple of order $p\pm 1$. Our computer experiments for the first $10,000$ primes show that such a triple can always be found in the orbit of $(3,3,3)$ under the group generated by $R_2R_2$, which consists of triples
$$(3,3F_{2n-1},3F_{2n+1}) \mbox{ for } n\geq 1,$$
modulo $p$, where $F_k$ denotes the $k$-th Fibonacci number.
\end{remark}

\begin{proposition}\label{thm:Td}Let $\tau_d(n)$ denote the number of divisors of $n$ that are $\leq d$. For $d$ dividing $p-1$ or $p+1$, let $T_d =\tau_d(p-1)+\tau_d(p+1)$. If no such divisor satisfies either inequality below: 
$$\frac{2\sqrt{2p}}{T_d}<d<\frac{81T_d^3}{4}\hspace{3cm}\frac{p}{6T_d}< d < \frac{8\sqrt{p}(p\pm 1)\tau(p\pm 1)}{\phi(p\pm 1)}$$ (where the $\pm$ is $+$ when $d|p+1$ and $-$ if $d|p-1$), then $\Gcal_p$ is connected.\end{proposition}

\begin{proof}Suppose $p$ is such that the Markoff graph $\text{mod}\,p$ is not connected, and let $d$ be the maximal order among triples that are not in $\mathcal C_p$. Fix some triple not in $\mathcal C_p$ that attains $d$ as the order of its first coordinate (without loss of generality), and write it in the form of (\ref{eq:7}).

By maximality of $d$ among orders in $\Gcal_p\backslash\Ccal_p$, each of second and third coordinates in the orbit (\ref{eq:7}) must have order $d'\leq d$, where $d'\,|\,p\pm 1$ as usual. There are exactly $d$ choices of exponent $n\,\text{mod}\,d$ in the second and third coordinates of (\ref{eq:rotorbit}), so with $\mathcal{T}_d$ denoting the set of divisors of $p\pm1$ that do not exceed $d$, Lemma~\ref{prop:corvaja} implies 
\begin{equation}\label{eq:9}d\leq \sum_{d'\in \mathcal{T}_d}\frac{3}{2}\max\!\left((6dd')^{1/3},\frac{4dd'}{p}\right) < \frac{3T_d}{2}\max\!\left((6d^2)^{1/3},\frac{4d^2}{p}\right).\end{equation}

First consider the case $\max((6d^2)^{1/3},4d^2/p)=4d^2/p$. Adding this to right-hand side above and solving for $d$ gives $d > p/6T_d$. A large divisor like this is amenable to the End Game in \cite{BGS}, so we apply Proposition~\ref{preciseBGSEG} in the Appendix to get $$\frac{p}{6T_d}< d < \frac{8\sqrt{p}(p\pm 1)\tau(p\pm 1)}{\phi(p\pm 1)},$$ as in the statement of this proposition. 

Next consider the case $\max((6d^2)^{1/3},4d^2/p)=(6d^2)^{1/3}$. Again use this with (\ref{eq:9}) and solve for $d$ to get $d<81T_d^3/4$; so it remains only to show $2\sqrt{2p}/T_d < d$ to complete the proof. To that end, the number of distinct $a\in \F_p\backslash\{\pm 2\}$ for which $\text{ord}_p(a)$ divides $d'$ is at most $d'/2$ (as $a=r+r^{-1}$ and $a=r^{-1}+(r^{-1})^{-1}$ should only be counted once). So we can bound the number of Markoff triples $(a,b,c)$ of order at most $d$ by summing over the different possible orders of $a$ and $c$ and noting that there are at most two choices for $c$ that produce a Markoff triple once $a$ and $b$ are fixed: \begin{equation}\label{eq:10}
\sum_{d',d''\in \mathcal{T}_d}\!\!\!\!2\cdot\frac{d'}{2}\cdot\frac{d''}{2} < \frac{T_d^2d^2}{2}.
\end{equation} Our choice of $d$ means $|\Gcal_p\backslash \mathcal C_p|$ cannot exceed the number of Markoff triples of order at most $d$. This allows us to combine (\ref{eq:10}) and Lemma \ref{lem:divisible}, giving $4p<T_d^2d^2/2$. Thus $2\sqrt{2p}/T_d < d$ as desired.\end{proof}

\begin{corollary}\label{cor:prelim}$\Gcal_p$ is connected for all primes $p > 10^{532}$.\end{corollary}

\begin{proof}First let us bound $T_d$ from Proposition~\ref{thm:Td} using Nicolas' upper bound on $\tau(n)$ \cite{nicolas}, which is $$\tau(n)< \text{exp}\!\left(\frac{\log 2\log n}{\log\log n}+\frac{1.342\log n}{(\log\log n)^2}\right).$$ This gives \begin{equation}\label{eq:11}T_d\leq\tau(p-1)+\tau(p+1)< 2\,\text{exp}\!\left(\frac{\log 2\log p}{\log\log p}+\frac{1.342\log p}{(\log\log p)^2}\right),\end{equation} where the final inequality has used that the function bounding $\tau(n)$ is concave in order to average the inputs $p-1$ and $p+1$. Now, to show that the first inequality in Theorem \ref{thm:Md} is never satisfied for $p > 10^{532}$, we will check that $81T_d/4 \leq 2\sqrt{2p}/T_d$ for all $d$. Rearranging this inequality slightly, taking the natural logarithm, and replacing $T_d$ with the bound in (\ref{eq:11}) gives $$2\log(81\sqrt{2})\leq\log p\left(1-\frac{8\log 2}{\log\log p}-\frac{10.736}{(\log\log p)^2}\right),$$ which is easily verified for $p > 10^{532}$. 

A similar approach shows that the second inequality in Proposition~\ref{thm:Td} is also never satisfied. Using the same bounds on $\tau(p\pm 1)$ along with $\phi(p\pm 1) > p/(2\log\log p)$ (a weaker version of Theorem 8.8.7 in \cite{bach}) shows that $8\sqrt{p}(p\pm 1)\tau(p \pm 1)/\phi(p\pm 1) \leq p/6T_d$ when $p > 10^{141}$.\end{proof}

\section{Maximal Divisors}\label{MaxDivsec}

We can improve the bound in Corollary~\ref{cor:prelim} by using the notion of what we call \emph{maximal divisors}. The key observation is that the count in Lemma~\ref{prop:corvaja} comes from counting the number of solutions in a subgroup of $\mathbb F_p^\ast$ of order $t$ to the equation in (\ref{czcurve}). So whenever we consider two divisors $t,t'<d$ of $p\pm1$ where $t|t'$, we count the solutions relevant to the divisor $t$ twice, since the subgroup of order $t$ is contained in that of the subgroup of order $t'$. So, instead of summing over all divisors in (\ref{eq:9}), we can sum over a refined set of divisors that we call maximal.

\begin{definition}\label{maxdivdefn}
Let $n$ be a positive integer, and let $x\in\R$. A positive divisor $d$ of $n$ is said to be \emph{maximal with respect to $x$} if $d\leq x$ and there is no other positive divisor $d'$ of $n$ such that $d'\leq x$ and $d\,|\,d'$. The set of maximal divisors with respect to $x$ is denoted $\Mcal_x(n)$.
\end{definition}

Our goal now is to improve on the bound in Corollary~\ref{cor:prelim} by replacing the set $\mathcal{T}_d$ with the set $\Mcal_d$ as shown in this simple improvement of Proposition~\ref{thm:Td}.

\begin{theorem}\label{thm:Md}For $d$ dividing $p-1$ or $p+1$, let $M_d =|\Mcal_d(p-1)|+|\Mcal_d(p+1)|$. If no such divisor satisfies either inequality below: $$\frac{2\sqrt{2p}}{M_d}<d<\frac{81M_d^3}{4}\hspace{3cm}\frac{p}{6M_d}< d < \frac{8\sqrt{p}(p\pm 1)\tau(p\pm 1)}{\phi(p\pm 1)}$$ (where the $\pm$ is determined by whether $d$ divides $p-1$ or $p+1$), then $\Gcal_p$ is connected.\end{theorem}

The proof of this is identical to that of Proposition~\ref{thm:Td}, replacing all instances of $T_d$ with $M_d$, and noting that the rotation order $d'$ of the second and third coordinates in the orbit (\ref{eq:7}) must divide at least one maximal divisor of $p\pm 1$ with respect to $d$.

In Section~\ref{firstbound}, we relied on known upper bounds for $\tau(n)$, and now we hope to obtain helpful bounds on $M_d$. There is very little in the literature on the number of maximal divisors of $n$ with respect to $x$. To find asymptotic and explicit bounds for small $n$, our strategy is to first find those $n$ for which $|\Mcal_x(n)|$ is maximized, akin to Ramanujan's ``superior highly composite numbers."

In \cite{ramanujan}, Ramanujan introduced a simple approach to bounding $\tau(n)$ in which only a very sparse set of integers $n$, which he called superior highly composite numbers, needs to be considered. They are those $n$ that maximize $\tau(n)/n^{\varepsilon}$ for some $\varepsilon > 0$. The prime factorization of a superior highly composite number was determined by Ramanujan to be $2^{a_1}3^{a_2}5^{a_3}\cdots$ where $$a_i=\left\lfloor\frac{1}{p_i^\varepsilon-1}\right\rfloor.$$ These numbers are convenient for two main reasons: First, they are easy to enumerate because the prime factorizations are known and there are fewer than $\log x$ superior highly composite numbers less than $x$ if $x>10^9$. Second, if $n_1$ and $n_2$ are consecutive superior highly composite numbers and $f$ is a convex function on the interval $(e^{n_1},e^{n_2})$, then $\log \tau(n)\leq f(\log n)$ holds for all integers $n\in [n_1,n_2]$ if and only if it holds for $n_1$ and $n_2$. These two facts make it easy to obtain both asymptotic bounds on $\tau(n)$ and a sharp bound on $\tau(n)$ in a given interval. Our goal in this section is to recreate this approach for $|\Mcal_x(n)|$ in place of $\tau(n)$.

\subsection{Reducing functions}\label{sec:reduce} In this section we introduce a tool for narrowing down the list of integers $n$ for which $|\Mcal_x(n)|$ needs to be computed to obtain upper bounds. Our work culminates in Definition \ref{def:redn} and Theorem \ref{thm:redn}.

\begin{notation}For $n\in\N$ let $\Dcal(n)$ denote the set of positive divisors of $n$, and let $\lambda(n)$ denote the least prime factor of $n$ if $n\geq 2$. Set $\lambda(1)=1$. \end{notation}

The function $\lambda$ is often denoted ``lpf" or ``LD" in the literature.

\begin{definition}\label{def:redf}For $m,n\in\N$, a function $f:\Dcal(n)\to\Dcal(m)$ is called \emph{reducing} if and only if the following hold for all $d,d'\in\Dcal(n)$: 

\vspace{0.2cm}

\begin{enumerate}[(a)]\setlength\itemsep{0.2cm}\item $f(d)\leq d$, \item $\displaystyle\frac{m/f(d)}{n/d}\leq \min\!\left\{1,\,\frac{\lambda(m/f(d))}{\lambda(n/d)}\right\}$, \item $f(d)=2^if(d')$ for some $i\in\Z$ implies $d=2^jd'$ for some $j\in\Z$.\end{enumerate}

\vspace{0.2cm}

\noindent We say $n$ \textit{reduces to} $m$ when such a function exists.\end{definition}

Observe that setting $d = n$ in requirement (b) results in $m/f(n)\leq 1$. Since $f(n)\,|\,m$, this forces $f(n) = m$, which combines with requirement (a) to give $m\leq n$. So integers can only reduce to smaller integers.

\begin{theorem}\label{thm:inject}If $n$ reduces to $m$ then $|\Mcal_x(n)|\leq |\Mcal_x(2^am)|$ for all $x\in\R$, where $a$ is the smallest integer satisfying $2^am\geq n$.\end{theorem}

\begin{proof}There is little to check if $x\geq n$, so assume otherwise. We claim that a reducing function $f:\Dcal(n)\to\Dcal(m)$ induces an injection $\hat{f}:\Mcal_x(n)\to\Mcal_x(2^am)$ defined by $\hat{f}(d)=2^if(d)$, where $i$ is the largest integer such that $2^if(d)\leq x$ and $2^if(d)\in\Dcal(2^am)$. Note that (a) in Definition \ref{def:redf} guarantees $i\geq 0$.

First let us verify that $\hat{f}(d)\in\Mcal_x(2^am)$. Since $\hat{f}(d)\leq x < n\leq 2^am$, we see that $\hat{f}(d)$ has proper multiples in $\Dcal(2^am)$, and it must be verified that they exceed $x$. That is, we must show $\hat{f}(d)\lambda(2^am/\hat{f}(d))>x$. This is immediate by maximality of $i$ if $\lambda(2^am/\hat{f}(d))$ happens to be $2$. Referring to the three inequalities below, the first follows from $\lambda(2^am/\hat{f}(d))\neq 2$, the second is a slight rearrangement of (b) in Definition \ref{def:redf}, and the third follows from our choice of $a$: $$\hat{f}(d)\lambda\!\left(\frac{2^am}{\hat{f}(d)}\right)=2^if(d)\lambda\!\left(\frac{2^am}{2^if(d)}\right)\geq 2^af(d)\lambda\!\left(\frac{m}{f(d)}\right)\geq \frac{2^amd}{n}\lambda\!\left(\frac{n}{d}\right)\geq d\lambda\!\left(\frac{n}{d}\right).$$ Since $d\in\Mcal_x(n)$ and $d$ properly divides $d\lambda(n/d)$ (recall that we are assuming $x < n$, so $d\neq n$), we must have $d\lambda(n/d)>x$ by definition of maximal divisors. Combined with the inequalities above, this completes our argument that $\hat{f}(d)\in\Mcal_x(2^am)$.

Next we check that $\hat{f}$ is an injection. If $\hat{f}(d)=\hat{f}(d')$ then $2^if(d)=2^{i'}\!f(d)$ for some $i,i'\in\Z$. This means $d=2^jd'$ for some $j\in\Z$ by (c) in Definition \ref{def:redf}, so either $d$ divides $d'$ or vice versa. But then $d,d'\in\Mcal_x(n)$ forces $d=d'$ by definition of maximal divisors.\end{proof}

In this last theorem, $2^am<2n$. So at the expense of less than a factor of $2$, we can forgo computing $|\Mcal_x(n)|$ in favor of computing $|\Mcal_x(2^am)|$, the hope being that $m$ has some kind of predictable prime factorization like the superior highly composite numbers.

Let us consider a simple example. If $p$ and $q$ are primes with $2\neq p\leq q$, then $f:\Dcal(q^a)\to\Dcal(p^a)$ defined by $f(q^i)=p^i$ is a reducing function. All three requirements from Definition \ref{def:redf} are trivially satisfied. Using $f$ to ``replace" $q^a$ with $p^a$ may not seem useful computationally because $|\Mcal_x(q^a)|$ just equals 1 for any $x$, but we can actually use $f$ to swap primes within a prime factorization. That is, if $n$ is not divisible by $p$ or $q$, then $f$ can be extended to a reducing function $\Dcal(nq^a)\to\Dcal(np^a)$ via the next lemma. 

\begin{lemma}\label{lem:prod}Suppose $n_1,n_2,m_1,m_2\in\N$ are such that $\emph{gcd}(n_1,n_2)=\emph{gcd}(m_1,m_2)=1$. If $f_1:\Dcal(n_1)\to\Dcal(m_1)$ and $f_2:\Dcal(n_2)\to\Dcal(m_2)$ are reducing then so is $f_1f_2:\Dcal(n_1n_2)\to\Dcal(m_1m_2)$.\end{lemma}

\begin{proof}Let $n=n_1n_2$, $m=m_1m_2$, and $f=f_1f_2$. Let $d,d'\in\Dcal(n)$, and let $d_1,d_1'\in\Dcal(n_1)$ and $d_2,d_2'\in\Dcal(n_2)$ be the unique divisors satisfying $d=d_1d_2$ and $d'=d_1'd_2'$. It is immediate that requirement (a) in Definition \ref{def:redf} holds for $f$ and that the ratio in requirement (b) is indeed bounded by 1. So let us turn our attention to the bound in (b) involving the $\lambda$ function. 

Suppose without loss of generality that $\lambda(m_1/f_1(d_1))\leq \lambda(m_2/f_2(d_2))$. Then 
\begin{eqnarray*}\frac{\lambda(m/f(d))}{\lambda(n/d)}&=&\frac{\min(\lambda(m_1/f_1(d_1)),\lambda(m_2/f_2(d_2)))}{\min(\lambda(n_1/d_1),\lambda(n_2/d_2))}\\
&=&\frac{\lambda(m_1/f_1(d_1))}{\min(\lambda(n_1/d_1),\lambda(n_2/d_2))}\\
&\geq&\frac{\lambda(m_1/f_1(d_1))}{\lambda(n_1/d_1)}\\
&\geq& \frac{m_1/f_1(d_1)}{n_1/d_1}\\
&\geq& \frac{m_1/f_1(d_1)}{n_1/d_1}\cdot\frac{m_2/f_2(d_2)}{n_2/d_2}\\
&=&\frac{m/f(d)}{n/d}.
\end{eqnarray*}

For requirement (c), suppose $f(d)=2^if(d')$ for some $i\in\Z$. Then $f_1(d_1)/f_1(d_1')=2^if_2(d_2')/f_2(d_2)$. By assumption, $\text{gcd}(f_1(d_1),f_2(d_2'))=\text{gcd}(f_1(d_1'),f_2(d_2))=1$, so $f_1(d_1)/f_1(d_1')$ and $f_2(d_2')/f_2(d_2)$ must be powers of $2$. Thus $d_1=2^{j_1}d_1'$ for some $j_1\in\Z$ because $f_1$ is reducing and $d_2=2^{j_2}d_2'$ for some $j_2\in\Z$ because $f_1$ is reducing. This gives $d=2^{j_1+j_2}d'$.\end{proof}

Returning to our example, if $p$ and $q$ do not divide some $n\in\N$, then Lemma \ref{lem:prod} allows us to combine our reducing function $\Dcal(q^a)\to\Dcal(p^a)$ with the identity $\Dcal(n)\to\Dcal(n)$ to obtain a reducing function $\Dcal(nq^a)\to\Dcal(np^a)$ in which $dq^i\mapsto dp^i$. That is, replacing larger primes with smaller ones in a prime factorization essentially produces no decrease in $|\Mcal_x(n)|$, as with the number of divisors function. The catch is the extra factor of $2$; in Theorem \ref{thm:inject}, $2^am$ can be almost twice as large as $n$. A natural concern is that with each successive maneuver like $q^a\mapsto p^a$, we pick up an extra factor of $2$. Knowing that $|\Mcal_x(n)|\leq |\Mcal_x(2^am)|$ from Theorem \ref{thm:inject} would not be helpful if $2^am$ was significantly larger than $n$. The next lemma eliminates that concern.

\begin{lemma}\label{lem:compose}If $f:\Dcal(n)\to\Dcal(m)$ and $g:\Dcal(m)\to\Dcal(\ell)$ are reducing, then so is $g\circ f$.\end{lemma}

\begin{proof}To see that $g\circ f$ satisfies requirement (b) in Definition \ref{def:redf}, we have \begin{eqnarray*}\frac{\ell/(g\circ f)(d)}{n/d}&=&\frac{\ell/(g\circ f)(d)}{m/f(d)}\cdot \frac{m/f(d)}{n/d}\\ &\leq&\min\!\left\{1,\,\frac{\lambda(\ell/(g\circ f)(d))}{\lambda(m/f(d))}\right\}\cdot\min\!\left\{1,\,\frac{\lambda(m/f(d))}{\lambda(n/d)}\right\}\\ &\leq&\min\!\left\{1\cdot 1,\,\frac{\lambda(\ell/(g\circ f)(d))}{\lambda(m/f(d))}\cdot\frac{\lambda(m/f(d))}{\lambda(n/d)}\right\}\\&=&\min\!\left\{1,\,\frac{\lambda(\ell/(g\circ f)(d))}{\lambda(n/d)}\right\}.\end{eqnarray*} Requirements (a) and (c) are immediate.\end{proof}

When combined, Lemmas \ref{lem:prod} and \ref{lem:compose} allow us to manipulate a prime factorization one comprehensible piece at a time. We have already seen through an example how to reduce to those $n$ whose $\omega(n)$ distinct prime factors are exactly $2,3,...,p_{\omega(n)}$. It turns out we can do even better: if $p$ and $q$ are primes with $2\neq p\leq q$ and $a$ and $b$ are integers with $0\leq a\leq b$, then there is a reducing function $f:\Dcal(p^aq^b)\to\Dcal(p^bq^a)$. It is defined by $f(p^iq^j)=p^{i+k}q^{j-k}$, where $k = \max(0,\min(i+j,b)-a)$. This allows us to rearrange prime exponents in decreasing order (except for the exponent of $2$). That is, to obtain bounds on $|\Mcal_x(n)|$, we need only consider those $n$ that are products of primorials up to a power of $2$. We will not prove that this function is reducing, because its purpose is subsumed by the next family of reducing functions. These not only rearrange exponents in decreasing order, they also limit the rate at which exponents can decrease.

\begin{lemma}\label{lem:expadd}Let $p$ and $q$ be distinct odd primes, let $a$ and $b$ be nonnegative integers, and set $c=\lfloor(a+1)/(b+2)\rfloor$. If $q< p^c$, then $p^aq^b$ reduces to $p^{a-c}q^{b+1}$.\end{lemma}

\begin{proof}Define $f:\Dcal(p^aq^b)\to\Dcal(p^{a-c}q^{b+1})$ by $f(p^iq^j)=p^iq^j$ if $i<(b+1-j)c$ and $f(p^iq^j)=p^{i-c}q^{j+1}$ if $i \geq (b+1-j)c$. We claim $f$ is a reducing function.

Suppose $i<(b+1-j)c$. The nontrivial assertion behind $f(p^iq^j)\in\Dcal(p^{a-c}q^{b+1})$ is that $i\leq a-c$. Indeed, $i\leq (b+1-j)c-1\leq (b+1)c-1=(b+2)c-c-1\leq (a+1)-c-1=a-c$. Requirements (a) and (c) are straightforward to check, so let us check (b), still in the case $f(p^iq^j)=p^iq^j$. We have $$\frac{p^{a-c}q^{b+1}/p^iq^j}{p^aq^b/p^iq^j}=\frac{q}{p^c}\leq \min\!\left\{1,\frac{q}{p}\right\}\leq \min\!\left\{1,\,\frac{\lambda(p^{a-c}q^{b+1}/p^iq^j)}{\lambda(p^aq^b/p^iq^j)}\right\}.$$

Next suppose $i \geq (b+1-j)c$. In this case it is clear that $f(p^iq^j)\in\Dcal(p^{a-c}q^{b+1})$. For requirement (b), $$\frac{p^{a-c}q^{b+1}/p^{i-c}q^{j+1}}{p^aq^b/p^iq^j}=1=\frac{\lambda(p^{a-c}q^{b+1}/p^{i-c}q^{j+1})}{\lambda(p^aq^b/p^iq^j)}.$$ Again, (a) and (c) are immediate in the case $i \geq (b+1-j)c$.\end{proof}

Next is a family of reducing functions devoted to controlling the exponent of $2$ in a prime factorization. Ultimately, $2$ will play the role of $p$ below. 

Both in the lemma statement and its proof, the empty product is to be interpreted as 1.

\begin{lemma}\label{lem:2exp}Let $p,q_1,...,q_k$ be primes with $p<q_1<\cdots<q_k$, and let $a\in\N$. If $p^{a-2} > q_1\cdots q_{k-1}q_k^2$ then $p^a$ reduces to $p^{b}q_1\cdots q_k$, where $$b=\left\lfloor\frac{1}{2}\left(a-\frac{\log(q_1\cdots q_{k-1})}{\log p}\right)\right\rfloor.$$\end{lemma}

\begin{proof}Let $c_k = a-b$ and $c_j=\lceil\log(q_1\cdots q_j)/\log p\rceil$ for $0\leq j<k$. We consider the case $j=k$ at the end of the proof. Define $f:\Dcal(p^a)\to\Dcal(p^{b}q_1\cdots q_k)$ by $f(p^{i}) = p^{b+c_j+i-a}q_{j+1}\cdots q_k$, where $j$ is the largest index such that $c_j\leq a-i$. We claim $f$ is a reducing function. 

The nontrivial assertion behind $f(p^i)\in\Dcal(p^bq_1\cdots q_k)$ is that $b+c_j+i-a\geq 0$. To verify this inequality, consider first the case $j<k-1$. The first inequality below follows from the choice of $j$, the second inequality uses the definitions of $c_j$ and $c_{j+1}$ (and assumes $j<k-1$), and the last inequality is the hypothesis $p^{a-2} > q_1\cdots q_{k-1}q_k^2$: 

\begin{eqnarray*}
b+c_j+i-a\geq b+c_j-c_{j+1}+1&\geq& b-\left\lceil\frac{\log q_{j+1}}{\log p}\right\rceil+1\\
&=&\left\lfloor\frac{1}{2}\left(a-\frac{\log(q_1\cdots q_{k-1})}{\log p}\right)\right\rfloor-\left\lceil\frac{\log q_{j+1}}{\log p}\right\rceil+1\\
&>&\frac{1}{2}\left(a-\frac{\log(q_1\cdots q_{k-1})}{\log p}\right)-\frac{\log q_{j+1}}{\log p}-1\\
&>&\frac{1}{2}\left(a-\frac{\log(q_1\cdots q_{k-1}q_k^2)}{\log p}\right)-1\\
&>& 0.
\end{eqnarray*}
In the case $j=k-1$ we must have $a-i\leq c_k-1=a-b-1$ by choice of $j$, so 
\begin{eqnarray*}
b+c_j+i-a&\geq& 2b+c_{k-1}+1-a\\
&=&2\left\lfloor\frac{1}{2}\left(a-\frac{\log(q_1\cdots q_{k-1})}{\log p}\right)\right\rfloor+\left\lceil\frac{\log(q_1\cdots q_{k-1})}{\log p}\right\rceil+1 -a\\ &>&2\left(\frac{1}{2}\left(a-\frac{\log(q_1\cdots q_{k-1})}{\log p}\right)-1\right)+\frac{\log(q_1\cdots q_{k-1})}{\log p}+1 -a\\
&=&-1.
\end{eqnarray*}
And finally, if $j=k$ then $b+c_j+i-a=i\geq 0$.

Now we turn to the bound $f(p^{i})\leq p^{i}$ from Definition \ref{def:redf}. If $j=k$ then $f(p^i)=p^i$. Otherwise,
\begin{eqnarray*}
\frac{\log (f(p^i)/p^i)}{\log p}&=&b+c_j-a+\frac{\log(q_{j+1}\cdots q_k)}{\log p}\\
&\leq& b-a+1+\frac{\log(q_1\cdots q_k)}{\log p}\\
&\leq&-\frac{1}{2}\left(a+\frac{\log(q_1\cdots q_{k-1})}{\log p}\right)+1+\frac{\log(q_1\cdots q_k)}{\log p}\\
&<&-\frac{1}{2}\left(2+\frac{2\log(q_1\cdots q_k)}{\log p}\right)+1+\frac{\log(q_1\cdots q_k)}{\log p}\\
&\leq& 0.
\end{eqnarray*}
To verify requirement (b), $$\frac{p^bq_1\cdots q_k/f(p^i)}{p^a/p^i}=\frac{q_1\cdots q_j}{p^{c_j}}\leq 1\leq \frac{\lambda(p^bq_1\cdots q_k/f(p^i))}{\lambda(p^a/p^i)},$$ where the final inequality above uses $p<q_1,...,q_k$. Requirement (c) is trivially satisfied.\end{proof}

Let us now identify those numbers that cannot be reduced by Lemma \ref{lem:expadd} or \ref{lem:2exp}. These are the numbers $n$ that we use to determine the maxima of $|\Mcal_x(n)|$, as made precise in Theorem \ref{thm:redn}. 

Throughout the remainder of this section, $p_i$ denotes the $i^\text{th}$ prime number.

\begin{definition}\label{def:redn}An integer $2^{a_1}3^{a_2}5^{a_3}\cdots$ (where $a_i=0$ for sufficiently large $i$) is \emph{reduced} if \begin{equation}\label{eq:1}\left\lfloor\frac{a_i+1}{a_j+2}\right\rfloor<\frac{\log p_j}{\log p_i}\end{equation} whenever $i,j\neq 1$, and $2^{a_1}<8p_j^2$ whenever $a_j=0$.\end{definition}

As examples, the first odd reduced numbers that are less than 100 are 1, 3, 9, 15, and 45. Up to a power of $2$, these numbers are products of primorials. This is always true, as mentioned before Lemma \ref{lem:expadd} and proved below. Also note the restriction on how quickly exponents can decrease. This is exhibited by the fact that 27 is not a reduced number---the exponent decrease from $3^3$ to $5^0$ is too much.

\begin{lemma}\label{lem:decrease}If $2^{a_1}3^{a_2}5^{a_3}\cdots$ is reduced, then $a_2\geq a_3\geq \cdots$.\end{lemma}

\begin{proof}On the one hand, if $i > j$ in inequality (\ref{eq:1}) then the right-hand side is less than 1. On the other hand if $a_i>a_j$ then the left-hand side is at least 1.\end{proof}

\begin{lemma}\label{lem:addone}Let $p_k$ be the largest prime divisor of $n\in\N$. If $n$ is reduced, so is $np_{k+1}$.\end{lemma}

\begin{proof}Only the exponent $a_{k+1}$ has changed, so we need only verify (\ref{eq:1}) when $i=k+1$ or $j=k+1$.

First suppose $i=k+1$ (so $a_i=1$ for $np_{k+1}$). If $j\leq k+1$ then $a_j\geq 1$ by Lemma \ref{lem:decrease} applied to $n$. Thus $$\left\lfloor\frac{a_i+1}{a_j+2}\right\rfloor \leq \left\lfloor \frac{2}{3}\right\rfloor = 0 < \frac{\log p_j}{\log p_i}.$$ If $j > k+1$ then $a_j = 0$. So $$\left\lfloor\frac{a_i+1}{a_j+2}\right\rfloor = 1 < \frac{\log p_j}{\log p_{k+1}}=\frac{\log p_j}{\log p_i}.$$

Now suppose $j=k+1$ and $i\neq k+1$. Here the fraction $(a_i+1)/(a_j+2)$ has decreased by adding the factor of $p_{k+1}$. So if inequality (\ref{eq:1}) holds for $n$, it certainly holds for $np_{k+1}$\end{proof}

\begin{theorem}\label{thm:redn}For any integer $n\geq 2$ there exists a reduced integer $m$ such that $n\leq m \leq 4n-6$ and $|\Mcal_x(n)|\leq|\Mcal_x(m)|$ for all $x\in\R$.\end{theorem}

\begin{proof}Let $m'$ be the odd part of the smallest positive integer to which $n$ can be reduced. By Lemma \ref{lem:expadd}, the exponents in the prime factorization of $m'$ satisfy (\ref{eq:1}). Let $a$ be the smallest integer such that $2^am'\geq n$. Then $|\Mcal_x(n)|\leq |\Mcal_x(2^am')|$ for all $x\in\R$ by Theorem \ref{thm:inject}. Note that $2^am'\leq 2n-2$. 

Let $p_k$ be the largest prime dividing $2^am'$, and if one exists, let $\ell$ be the largest index satisfying $p_{k+1}\cdots p_{\ell-1}p_\ell^2 < 2^{a-2}$. If no such index exists, let $\ell=k$. We claim that $m=2^{a_1}m'p_{k+1}\cdots p_\ell$ meets our theorem's requirements, where $a_1$ is the smallest integer such that $m\geq 2^am'$. From another application of Theorem \ref{thm:inject}, this time applied to the reduction in Lemma \ref{lem:2exp}, we have $|\Mcal_x(2^am')|\leq |\Mcal_x(m)|$ for all $x\in\R$. Since $$m\leq 2(2^am')-2\leq 2(2n-2)-2= 4n-6,$$ we will be done provided $m$ is reduced.

Apply Lemma \ref{lem:addone} $\ell-k$ times beginning with the reduced integer $m'$ to see that $m'p_{k+1}\cdots p_{\ell}$ is reduced, meaning (\ref{eq:1}) holds. Let us check that $2^{a_1}<8p_{\ell+1}^2$. We have $$3+\left\lfloor\frac{2\log p_{\ell+1}}{\log 2}\right\rfloor+\frac{\log(p_{k+1}\cdots p_\ell)}{\log 2} > 2+\frac{\log(p_{k+1}\cdots p_\ell\, p_{\ell+1}^2)}{\log 2}\geq a,$$ where the last inequality above uses maximality of $\ell$. Thus $3+\lfloor2\log p_{\ell+1}/\log 2\rfloor$ solves the inequality for which $a_1$ is the minimal solution, implying $a_1 < 3+2\log p_{\ell+1}/\log 2$ as desired.\end{proof}

Reduced numbers turn out to be sufficiently rare for our purpose. Data up to $x\approx 10^{10000}$ suggests that $12\log x$ is a very good approximation for the number of reduced $n\leq x$. This density could potentially be diminished further via new reducing functions, though the authors suspect that Definition \ref{def:redf} is too restrictive to allow for a notion of reduced numbers with density approaching that of the superior highly composite numbers (less than $\log x$ for large $x$). Definition \ref{def:redf} might be loosened, however, to permit functions $f:\Dcal(n)\to\Dcal(m)$ with ratios $$\alpha\coloneqq\max_{d\in\Dcal(n)}\frac{f(d)}{d}\hspace{0.5cm}\text{and}\hspace{0.5cm}\beta\coloneqq\max_{d\in\Dcal(n)}\frac{(m/f(d))\lambda(n/d)}{(n/d)\lambda(m/f(d))}$$ that exceed 1. Then, as long as $\alpha\leq \beta$, we could prove a version of Theorem \ref{thm:inject} that requires $2^am\geq \beta n$ in order to conclude $|\Mcal_x(n)|\leq |\Mcal_{\alpha x}(2^a m)|$ for all $x$.

\subsection{An asymptotic bound}

Our strategy for bounding $|\Mcal_x(n)|$ asymptotically is as follows: We need only consider reduced $n$ -- that is the purpose of the last section -- and reduced integers are not too far from being products of one or two primorials (Lemma \ref{lem:cubes}). This makes $\Omega(n)$ roughly equal to $\log n/\log\log n$ (Lemma \ref{lem:Omega}). If $x=n^\alpha$ then we expect elements of $\Mcal_x(n)$ to be products of roughly $\alpha\Omega(n)$ primes (Lemma \ref{lem:combo}), so we just apply Stirling's formula to bound how many ways we can choose these primes (Theorem \ref{thm:asym}).

\begin{lemma}\label{lem:cubes}For a reduced integer $2^{a_1}3^{a_2}\cdots p_k^{a_k}$, $$\sum_{a_i\geq 3}(a_i-2) = O\!\left(\frac{k^{2/3}}{(\log k)^{1/3}}\right).$$\end{lemma}

\begin{proof}By setting $j$ in Definition \ref{def:redn} equal to $k+1$, we see that $a_i<2\log p_{k+1}/\log p_i$ for any $i\geq 2$, and that $a_1 < 3+2\log p_{k+1}/\log 2$. In particular, if $a_i\geq 3$ then $p_i<p_{k+1}^{2/3}$. Let $x=p_{k+1}^{2/3}$. Our established inequalities followed by partial summation gives 
\begin{eqnarray*}
\sum_{a_i\geq 3}(a_i-2) &<& 3+2\!\sum_{p_i<x}\!\left(\frac{\log p_{k+1}}{\log p_i}-1\right)\\
&=& 3+2\pi(x)\!\left(\frac{\log p_{k+1}}{\log x}-1\right)+\int_2^x\!\frac{\pi(t)\log p_{k+1}}{t(\log t)^2}dt\\
&=& O(\pi(x)).
\end{eqnarray*}
Replacing $x$ with $p_{k+1}^{2/3}$ and applying the prime number theorem up to a constant multiple completes the proof.\end{proof}

A small deficiency in our reducing functions from Section \ref{sec:reduce} is that they do nothing to bound the index at which prime exponents of a reduced number must switch from 2 to 1. In fact, reduced numbers can be perfect squares. This is why the previous lemma can only bound sums of exponents that are at least 3 rather than at least 2, and thus why the proof of the next lemma must consider products of two primorials instead of a single primorial.

\begin{lemma}\label{lem:Omega}Let $\Omega(n$ denote the number of prime factors of $n$, counted with multiplicity. For a reduced integer $n$, $$\Omega(n)= \frac{\log n}{\log\log n}+O\!\left(\frac{\log n}{(\log\log n)^2}\right).$$\end{lemma}

\begin{proof}Suppose $n$ is reduced, and let $m$ be the largest factor of $n$ that is cube-free. So $m=p_k\# p_j\#$ for some $j\leq k$, where $p_j\#$ can be deleted if $m$ happens to be a primorial. 

We have two initial claims: $$\log n > (k+j)\log (k\log k)-3k$$ and (the crude bound) $$\log\log n < 2\log(k\log k),$$ both when $n$ and thus $k$ are large. To prove each of them, we will use standard bounds on Chebyshev's theta function, $$k(\log (k\log k)-1) < \vartheta(p_k) < k\log (k\log k)$$ (and similarly for $\vartheta(p_j)$ if $j$ is not bounded by some absolute constant) \cite{dusart}. First, we have 
\begin{align}\label{eq:2}\log n \geq \vartheta(p_k)+\vartheta(p_j) &>  k(\log (k\log k)-1) +  j(\log (j\log j)-1)\nonumber\\ &=(k+j)\log (k\log k)-(k+j) + j\log\!\left(\frac{j\log j}{k\log k}\right).\end{align} The smaller terms in the final expression are bounded multiples of $k$: \begin{equation}\label{eq:3}k+j\leq 2k,\hspace{0.5cm}\text{and}\hspace{0.5cm} -j\log\!\left(\frac{j\log j}{k\log k}\right)< \frac{k}{e}\left(1+\frac{1}{\log j}\right)<k.\end{equation} Combining (\ref{eq:2}) and (\ref{eq:3}) shows that $\log n > (k+j)\log (k\log k)-3k$ as desired. For the second claim, we have $\Omega(n/m)<k^{2/3}$ by Lemma \ref{lem:cubes}, so 
\begin{align}\label{eq:4}\log\log n &= \log (\log(n/m)+\log m)\nonumber\\&\leq \log(k^{3/2}\log p_k + \vartheta(p_k)+\vartheta(p_j))\nonumber\\ &< \log(3\vartheta(p_k))\nonumber\\&< 2\log(k\log k).\end{align}

Now we can combine our two initial claims as follows: \begin{eqnarray*}\frac{\Omega(n)\log\log n}{\log n} &<& \frac{(k^{2/3}+k+j)\log\log n}{\log n}\\ &<&\frac{(k^{2/3}+k+j)\log((k+j)\log (k\log k)-3k)}{(k+j)\log (k\log k)-3k}\\&<&\frac{(k+j)\log (k\log k)+2k}{(k+j)\log (k\log k)-3k}\\ &=& 1+O\!\left(\frac{1}{\log(k\log k)}\right)\\ &=& 1+O\!\left(\frac{1}{\log\log n}\right).\end{eqnarray*} Scaling both ends of the inequality above by $\log n/\log\log n$ completes the proof.\end{proof}

The notation below and the lemmas that follow it are purely combinatorial. We phrase them in the language of divisors for convenience.

\begin{notation}\label{not:Ck}For $n,k\in\Z$ with $n\geq 1$, let $C_k(n)=|\{d\in\Dcal(n):\Omega(d)=k\}|$.\end{notation}

So $C_k(n)$ counts the $k$-element multisets of the $\Omega(n)$-element multiset consisting of the prime factors of $n$ with multiplicity. In particular, if $n$ is square-free then $C_k(n)$ is just a binomial coefficient.

\begin{lemma}\label{lem:multiset} For any $n\in\N$, if $k\leq \Omega(n)/2$ then $C_{k-1}(n)\leq C_k(n)$. If $k\geq \Omega(n)/2$ then $C_k(n)\geq C_{k+1}(n)$.\end{lemma}

\begin{proof}In \cite{bruijn} it is shown that $\Dcal(n)$ can be partitioned into ``symmetric chains" of the form $\{d_1,...,d_j\}$, where $\Omega(d_1)+\Omega(d_j)=\Omega(n)$ and $\Omega(d_{i+1})=\Omega(d_i)+1$ for all $i=1,...,j-1$. So the multiset $\{\Omega(d):d\in\Dcal(n)\}$ is a disjoint union of sequences of consecutive integers, each centered at $\Omega(n)/2$.\end{proof}

\begin{lemma}\label{lem:combo}Given $n\in\N$ and $x\geq 1$, let $k$ be an integer that is closest to $\Omega(n)/2$ in the range $$\min\{\Omega(d):d\in\Mcal_x(n)\}\leq k\leq \max\{\Omega(d):d\in\Mcal_x(n)\}.$$ Then $|\Mcal_x(n)|\leq C_k(n).$\end{lemma}

\begin{proof}Again we partition $\Dcal(n)$ into symmetric chains $\{d_1,...,d_j\}$ as in the proof of Lemma \ref{lem:multiset}. Since elements of $\Mcal_x(n)$ cannot divide one another while elements of a particular symmetric chain always divide one another, each symmetric chain contains at most one maximal divisor. This allows us to define an injection from $\Mcal_x(n)$ to $\{d\in\Dcal(n):\Omega(d)=k\}$, and the latter multiset has cardinality $C_k(n)$. Indeed, to each $d\in\Mcal_x(n)$ we associate the unique divisor $d'$ that belongs to the same symmetric chain as $d$ and satisfies $\Omega(d')=k$. Such a $d'$ always exists because we chose $k$ to be at least as close to $\Omega(n)/2$ as $\Omega(d)$, and $\Omega(n)/2$ is the ``center" over which symmetric chains are symmetric.\end{proof}

We can now prove the asymptotic bound on $|\Mcal_x(n)|$ stated in the introduction.

\begin{customthm}{1.3}\label{thm:asym}For any $\varepsilon > 0$, if $\alpha\in[\varepsilon,1-\varepsilon]$ then $$\log|\Mcal_{n^\alpha}(n)|=\log\!\left(\frac{1}{\alpha^\alpha(1-\alpha)^{1-\alpha}}\right)\!\frac{\log n}{\log\log n}+O\!\left(\frac{\log n}{(\log\log n)^2}\right).$$ The implied constant depends only on $\varepsilon$.\end{customthm}

\begin{proof}Recall from Theorem \ref{thm:redn} that an integer $n$ can be replaced with a reduced integer at most four times its size. Since the increase from $\log n/\log\log n$ to $\log 4n/\log\log 4n$ is absorbed by the error term above, we need only prove this theorem for reduced integers. So let $n=2^{a_1}3^{a_2}\cdots p_k^{a_k}$ be reduced. 

Suppose first that $\alpha\geq 1/2$. Let $d_0$ be the divisor of $n$ such that $d_0\geq n^{\alpha}$, and $\Omega(d_0)$ is minimal among all divisors exceeding $n^\alpha$. Note that $d_0$ is composed of the largest primes dividing $n$, so $\Omega(d_0)\leq \alpha\Omega(n)$. This gives $$\Omega(n/d_0)\geq (1-\alpha)\Omega(n)\geq\varepsilon\Omega(n)\geq \varepsilon (k-1)$$ (note that $a_1$ might equal 0). Since $\varepsilon$ is fixed, if $n$ is sufficiently large then Lemma \ref{lem:cubes} implies $n/d_0$ must divisible by more primes than just those whose exponent in the factorization of $n$ exceeds 2. In particular, we see that $d_0$ is not divisible by any perfect cubes. That is, $d_0 = p_k\#p_j\#/(p_i\#)^2$ for some $i\leq j$. 

Since $\lambda(n/d)d > n^{\alpha}$ for any $d\in\Mcal_{n^\alpha}(n)$, the definition of $d_0$ implies $\Omega(d)+1\geq \Omega(d_0)$ for any $d\in\Mcal_{n^\alpha}(n)$. So our goal is to bound $\Omega(d_0)-1$ from below. To this end, the exact same argument from inequalities (\ref{eq:2}) and (\ref{eq:3}) shows that $$\log(n)>(k+j)\log (k\log k)-3k$$ for large $n$, and a nearly identical argument shows that $$\log d_0 \leq (k+j-2i-1)\log (k\log k)+3k$$ for large $n$. These are the first and third inequalities below, while the fourth uses Lemma~\ref{lem:cubes}:
\begin{eqnarray}\label{eq:5}\Omega(d_0)-1=k+j-2i-1 &>& \frac{\log d_0-3k}{\log(k\log k)}\nonumber \\
&>& \frac{\alpha\log n-3k}{\log(k\log k)}\nonumber\\
&>&\alpha(k+j)-\frac{3(1+\alpha)k}{\log(k\log k)} \nonumber\\
&=& \alpha(k+j+k^{2/3})\!\left(1-\frac{3(1+\alpha)k+\alpha k^{2/3}\log(k\log k)}{\alpha(k+j+k^{2/3})\log(k\log k)}\right)\\
&>& \alpha\Omega(n)\!\left(1-\frac{10}{\log(k\log k)}\right).\nonumber
\end{eqnarray}
Note that $\alpha\geq 1/2$ to justify the constant 10 for large $k$ in the final error term. Now recall from (\ref{eq:4}) that $\log(k\log k)$ can be replaced with $(\log\log n)/2$ above. In particular if $\beta\in\R$ is such that $\beta\Omega(n)$ is the closest integer to $\Omega(n)/2$ between $\min\{\Omega(d):d\in\Mcal_x(n)\}$ and $\max\{\Omega(d):d\in\Mcal_x(n)\}$ then \begin{equation}\label{eq:6}\beta > \alpha\left(1-\frac{20}{\log\log n}\right).\end{equation} 

Lemma \ref{lem:combo} followed by Stirling's formula tells us $$|\Mcal_{n^\alpha}(n)|\leq C_{\beta\Omega(n)}(n)\leq \binom{\Omega(n)}{\beta\Omega(n)}=\Omega(n)^{O(1)}\left(\frac{1}{\beta^\beta(1-\beta)^{1-\beta}}\right)^{\!\Omega(n)}\!\!.$$ Now let $f(x) = (x-1)\log(1-x)-x\log x$, and take logarithms of the inequalities above to get
\begin{eqnarray*}
    \log|\Mcal_{n^{\alpha}}(n)| &=& O(\log\Omega(n)) + f(\beta)\Omega(n)\\
    &=& f(\beta)\!\left(\frac{\log n}{\log\log n}+O\!\left(\frac{\log n}{(\log\log n)^2}\right)\right)\\
    &\leq& \left(f(\alpha) + \frac{20\alpha|f'(\alpha)|}{\log\log n}\right)\!\left(\frac{\log n}{\log\log n}+O\!\left(\frac{\log n}{(\log\log n)^2}\right)\right)\\
    &=& f(\alpha)\frac{\log n}{\log\log n} + O\!\left(\frac{\log n}{(\log\log n)^2}\right).
\end{eqnarray*}
Both the second and last equality above use that $\alpha$ (and $\beta$) are restricted to the interval $[\varepsilon,1-\varepsilon]$. The lone inequality symbol above is justified by (\ref{eq:6}) and the mean value theorem.

We need not repeat these arguments for $\alpha < 1/2$. Indeed, the only missing piece is an analogous upper bound on $\Omega(d_1)$, where $d_1$ is the divisor of $n$ such that $d_1\leq n^{\alpha}$, and $\Omega(d_1)$ is maximal among all divisors not exceeding $n^\alpha$. But this makes $d_1=n/d_0$. So by (\ref{eq:5}), but with $\alpha$ replaced by $1-\alpha$, we have $$\Omega(d_1) = \Omega(n)-\Omega(d_0) < \Omega(n)\!\left(1-(1-\alpha)\!\left(1-\frac{20}{\log\log n}\right)\right) = \alpha\Omega(n)\!\left(1+O\!\left(\frac{\log n}{\log\log n}\right)\right).$$ The uses of Stirling's formula and the mean value theorem work again with trivial modification.\end{proof}

As a corollary, we get an asymptotic bound on the total number of divisors of $n$ bounded by $n^{\alpha}$.

\begin{corollary}\label{cor:asym}For any $\varepsilon > 0$, if $\alpha\in[\varepsilon,1/2]$ then $$\log\big|\{d\in\Z:d\,|\,n,\,d\leq n^\alpha\}\big|=\log\!\left(\frac{1}{\alpha^\alpha(1-\alpha)^{1-\alpha}}\right)\!\frac{\log n}{\log\log n}+O\!\left(\frac{\log n}{(\log\log n)^2}\right).$$ The implied constant depends only on $\varepsilon$.\end{corollary}

\begin{proof}Let $x\in\R$, and suppose $d$ is a proper divisor of $n$ in $(x/2,x]$. Since $\lambda(n/d)\geq 2$, we see that $d\lambda(n/d)>x$, implying $d\in\Mcal_x(n)$. Therefore to cover the entire set of divisors in the corollary statement, it suffices to union only the sets $\Mcal_x(n)$ for $x=\lfloor n^\alpha\rfloor,\lfloor n^\alpha/2\rfloor,...,1$. There are at most $\lfloor \log_2n^\alpha\rfloor+1$ such values of $x$. By Lemmas \ref{lem:multiset} and Lemma \ref{lem:combo}, each $|\Mcal_x(n)|$ is bounded by $C_{\Omega(d_1)}(n)$, where $d_1$ is as in the previous proof: the divisor of $n$ such that $d_1\leq n^{\alpha}$, and $\Omega(d_1)$ is maximal among all divisors not exceeding $n^\alpha$. We just showed that $$\log C_{\Omega(d_1)}(n) =\log\!\left(\frac{1}{\alpha^\alpha(1-\alpha)^{1-\alpha}}\right)\!\frac{\log n}{\log\log n}+O\!\left(\frac{\log n}{(\log\log n)^2}\right),$$ and scaling $C_{\Omega(d_1)}(n)$ by $\lfloor\log_2n^\alpha\rfloor+1$ does not change this.\end{proof}

As mentioned in the introduction, when $\alpha=1/2$ Corollary \ref{cor:asym} recovers Wigert's theorem that $\log\tau(n)=(\log 2+o(1))(\log n/\log\log n)$ \cite{Wigert}.


\section{Proof of Theorem~\ref{thm:main}}\label{prooftheorem}

Further reduction to the preliminary bound of $p>10^{532}$ from Corollary \ref{cor:prelim} can now be obtained with maximal divisors. We aim to determine more precisely the minimal value of $p$ needed to guarantee the first interval in Theorem \ref{thm:redn} is empty. The second interval in Theorem~\ref{thm:Md} is ignored -- it is empty for $p>10^{141}$ as shown in the proof of Corollary \ref{cor:prelim}. This is much smaller than what we might hope to work for the first interval.

Let us give an intuitive outline of how Algorithm \ref{alg} works. Recalling Theorem \ref{thm:Md}, the first interval is empty precisely when $81M_d^4<8\sqrt{2p}$. To determine when this occurs we need upper bounds on $$M_d\coloneqq |\Mcal_d(p-1)\cup\Mcal_d(p+1)|$$ for varying $d$ and $p<10^{532}$. There are roughly $10^{529}$ such primes, so of course we cannot hope to treat them individually. Instead we apply Theorem \ref{thm:redn}, which says we can obtain bounds on $M_d$ by bounding $|\Mcal_d(n)|$ for all reduced $n$ between $p$ and $4p-2$. There are only 16,899 reduced numbers less than $4\cdot 10^{532}$, which is much more manageable.

For a reduced number $n$, Algorithm \ref{alg} begins by using Lemma \ref{lem:combo} to find an upper bound $C$ on $2|\Mcal_x(n)|$ that applies regardless of $x$. (The ``2" accounts for $p-1$ and $p+1$.) But then if $M_d<C$, we realize from the first inequality in Theorem \ref{thm:Md} that we actually only need a bound on $M_d$ that applies when $d < 81C^3/4$. So we use Lemma \ref{lem:combo} again to compute a potentially smaller $C$ that need only apply in this reduced range of $x$. The hope is to reduce $C$ until the first interval in Theorem \ref{thm:Md} that might contain $d$ is empty because $81C^3/4 \leq 2\sqrt{p}/C$.

\begin{algorithm}\label{alg}
    \DontPrintSemicolon
    \KwIn{$a,b\in\N$ defining the range $(a,b]$ in which primes are tested}
    \KwOut{$a,b$ with $a$ updated so that $\mathcal{G}_p$ is connected if $a<p\leq b$}
    \For(\commentf{$\triangleright\;$see Definition \ref{def:redn}; order doesn't matter}){\textup{reduced }$n$\textup{ from }$a$\textup{ to }$4b-2$}
    {
        $k\gets \lfloor\Omega(n)/2\rfloor$\commentr{$\triangleright\;2C_k(n)$ bounds $M_d$ from Theorem \ref{thm:Md}}
        \While(\commentf{$\triangleright\;$Theorem \ref{thm:Md}'s first interval not empty...}){$n+2<8(3C_k(n))^8$}
        {
            $j\gets\max\{\Omega(d):d\,|\,n,\,d<162C_k(n)^3\}$\;
            \If(\commentf{$\triangleright\;$...and it never will be}){$j\geq k$}
            {
                $a\gets n+1$\commentr{$\triangleright\;$connectivity test failed for $p\leq a$}
                \textbf{break}\;
            }
            $k\gets j$\;
        }
    }
    \Return{$a,b$}\commentr{$\triangleright\;$empty first interval if $a<p\leq b$}
    \caption{Connectivity test for Markoff mod-$p$ graphs for all primes in a given interval.}
\end{algorithm}

\begin{theorem}If $a$ and $b$ are outputs of Algorithm \ref{alg} and $a<p\leq b$, then $\Gcal_p$ is connected.\end{theorem}

\begin{proof}Suppose $p$ is a prime for which $\mathcal{G}_p$ is not connected. Assuming $a<p\leq b$, we must show that $p$ is at most the output value of $a$. By Theorem \ref{thm:Md} there is a divisor, call it $d_0\in\Dcal(p+1)\cup\Dcal(p-1)$, such that \begin{equation}\label{eq:}\frac{2\sqrt{2p}}{M_{d_0}}<d_0<\frac{81M_{d_0}^3}{4},\end{equation} where $M_{d_0}=|\Mcal_{d_0}(p-1)\cup\Mcal_{d_0}(p+1)|$. 

Let $n_{\pm}$ be the reduced integers provided by Theorem \ref{thm:redn} for $p\pm1$. According to Theorem \ref{thm:redn}, $$p\pm 1\leq n_{\pm}\leq 4(p\pm 1)-6,$$ which in turn gives $a\leq n_{\pm}\leq 4b-2$ since $a<p\leq b$. So at some point(s) in Algorithm \ref{alg}'s \textbf{for} loop, $n$ will assume the value of $n_{-}$ and $n_+$. 

Assume without loss of generality that $|\Mcal_{d_0}(n_+)|\geq |\Mcal_{d_0}(n_-)|$. Call $k\in\N$ \textit{sufficiently large} if it is at least as close to $\Omega(n_+)/2$ as anything between $$\min\{\Omega(d):d\in\Mcal_{d_0}(n_+)\}$$ and $$\max\{\Omega(d):d\in\Mcal_{d_0}(n_+)\}.$$ Lemmas \ref{lem:multiset} and \ref{lem:combo} tell us that $|\Mcal_{d_0}(n_+)|\leq C_k(n_+)$ for such $k$. Thus $$M_{d_0}\leq |\Mcal_{d_0}(p-1)|+|\Mcal_{d_0}(p+1)|\leq |\Mcal_{d_0}(n_-)|+|\Mcal_{d_0}(n_+)|\leq 2|\Mcal_{d_0}(n_+)|\leq 2C_k(n_+).$$ This combines with (\ref{eq:}) to give $$n_+\!+2\leq 4p < (3M_{d_0})^8/32\leq 8(3C_k(n_+))^3.$$ Note that the first inequality uses the upper bound on $n_+$ from Theorem \ref{thm:redn}, which also holds if $n_-$ is used instead. So the \textbf{while} loop condition in line 3 is always satisfied if $k$ is sufficiently large. 

Now, by induction on the number of \textbf{while} loop iterations completed for $n_+$, the value of $k$ used in line 3 is always sufficiently large. Indeed, the base case holds by line 2. And for the induction step, either $j$ from line 4 is at least $\lfloor\Omega(n_+)/2\rfloor$ (in which case the \textbf{while} loop terminates by lines 5 and our proof is complete by line 6), or $j$ is sufficiently large because \begin{eqnarray*}j&=&\max\{\Omega(d):d\,|\,n_+,\,d<162C_k(n_+)^3\}\\&\geq&\max\{\Omega(d):d\,|\,n_+,\,d<81M_{d_0}^3/4\}\\&\geq&\max\{\Omega(d):d\,|\,n_+,\,d\leq d_0\}\\&=&\max\{\Omega(d):d\,|\,n_+,\,d\in\Mcal_{d_0}(n_+)\}.\end{eqnarray*} Thus the \textbf{while} loop continues to iterate until the \textbf{if} condition in line 5 is met, which happens eventually since $k$ cannot decrease indefinitely. So by line 6, the output satisfies $a\geq n_{\pm}+1\geq p$.\end{proof}

Finally, we use Algorithm \ref{alg} to produce our main result.

\begin{customthm}{1.4}$\Gcal_p$ is connected for all primes $p > 863\#53\#13\#7\#5\#3^32^5\approx 3.448\cdot10^{392}$.\end{customthm}

\begin{proof}By Corollary \ref{cor:prelim}, we need only check connectivity for primes less than $10^{532}$. When $a=2$ and $b=10^{532}$ are input into Algorithm \ref{alg}, the output is $a=863\#53\#13\#7\#5\#3^32^5+1$. Since this number is not prime, the ``$+1$" has been omitted in the theorem statement.\end{proof}

The prime $p=863\#53\#13\#7\#5\#3^32^5-1471$ is the largest for which we do not know whether $\Gcal_p$ is connected.

\section{Data on Connectivity}\label{data}

Aside from justifying Algorithm \ref{alg}, Theorem~\ref{thm:Md} also provides a method for verifying connectivity of $\Gcal_p$ for a given prime $p$. Previously, proving connectivity for $\Gcal_p$ has been done in \cite{dCL} for primes less than $3000$ by computing the adjacency matrix of the graph. Due to the large amount of memory required by this method, it has limitations as to how large a prime it could handle. Most likely one could not prove connectivity for primes  larger than a few thousand using this method. Our algorithm, on the other hand, is specifically catered towards larger primes (and, indeed, is inconclusive for nearly all the primes handled in \cite{dCL}). In this section, we prove connectivity for many more primes and explore how powerful our method is regarding the size of a primes that it can handle. 

 We programmed the two conditions of Theorem~\ref{thm:Md} and performed an exhaustive search over all primes less than $10^7$ that satisfy these conditions. We found that Theorem~\ref{thm:Md} proves connectivity for $p=3, 7, 101$ and then the next prime is on the order of $10^6$, given by 
    \begin{align*}
        p=1,327,363.
    \end{align*}
    After finding this first prime with a connected Markoff mod-$p$ graph that was not handled by \cite{dCL}, we tackled two collections of primes: the first 10000 primes greater than $10^n$ and 10000 ``random" primes between $10^n$ and $10^{n+1}$ for $8 \leq n \leq 35$. By random primes, we mean that we take $10000$ numbers between $10^n$ and $10^{n+1}$ chosen uniformly at random, and then for each number find the first prime greater than it.
\begin{table}[ht!]
    \centering
    \begin{tabular}{|c|c|c|c|}
    \hline
        $n$ & $q_{1000}(10^n)$ & $q_{10000}(10^n)$ & $r_{10000}(10^n)$\\ \hline
         8 &  21.3\% & 20.22\% & 38.12\%\\
         9 &  48.1\% & 49.04\% & 67.46\% \\
         10 &  76.1\% & 76.41\% & 87.05\% \\
         11 & 90.9\% & 90.78\%  & 95.33\% \\
         12 & 96.6\% & 97.10\% & 98.29\% \\ 
         13 & 98.8\% & 98.65\% & 99.11\%\\
         14 & 99.4\% & 99.44\%& 99.52\% \\
         15 & 99.7\% & 99.74\% & 99.83\% \\
         16 & 99.7\% & 99.88\% & 99.88\%\\
         17 & 99.9\% & 99.93\%& 99.95\% \\
         18 & 100\% & 99.97\% & 100\%\\
         19 & 100\% & 99.97\% & 99.97\% \\
         20 & 99.8\% & 99.97\% & 100\%\\
         21 & 100\% & 99.99\% & 99.99\%\\
         \hline
    \end{tabular} \qquad \qquad
        \begin{tabular}{|c|c|c|c|}
        \hline
        $n$ & $q_{1000}(10^n)$ & $q_{10000}(10^n)$ & $r_{10000}(10^n)$\\ \hline
         22 & 100\% & 100\% & 100\%\\
         23 & 100\% & 100\% & 100\% \\
         24 & 100\% & 100\% & 100\% \\
         25 & 100\% & 100\% & 100\% \\
         26 & 100\% & 100\% & 100\% \\
         27 & 100\% & 100\% & 100\% \\
         28 & 100\% & 100\% & 100\% \\
         29 & 100\% & 100\% & 100\% \\
         30 & 100\% & 100\% & 100\% \\
         31 & 100\% & 100\% & 100\% \\
         32 & 100\% & 100\% & 100\% \\
         33 & 100\% & 100\% & 100\% \\
         34 & 100\% & 100\% & 100\% \\
         35 & 100\% & 100\% & 100\% \\ 
         \hline
    \end{tabular}
    \vspace{2mm}
    \caption{For each value of $8 \leq n \leq 35$, we calculate the two quantities $q_m(10^n)$ and $r_m(10^n)$. $q_{m}(10^n)$ denotes the percentage of the first $m$ primes after $10^n$ for which Theorem \ref{thm:Md} guarantees connectivity of $\Gcal_p$ and $r_{m}(10^n)$ denotes the percentage of $m$ random primes between $10^n$ and $10^{n+1}$ for which Theorem \ref{thm:Md} guarantees connectivity of $\Gcal_p$.
    }
    \label{tab:percentage}
\end{table}

Beginning at $n=31$ in the table above, the value of $M_d$ in Theorem \ref{thm:Md} can be replaced with $\tau(p-1)+\tau(p+1)$ (which can be computed quickly for primes up to at least $10^{90}$), and there is still no value of $d$ satisfying either of the inequalities for the 10,000 random primes we tested between $10^n$ and $10^{n+1}$. That is, $10^{31}$ is roughly where the Erd\"{o}s-Kac theorem takes over---the expected value of $\tau(p\pm1)$ is small enough so that it becomes extremely rare to need the improvement that comes by considering maximal divisors rather than all divisors.

\vspace{3mm}
\noindent \textbf{Example of Inconclusiveness}: Theorem \ref{thm:Md} guarantees connectedness of the Markoff mod $p$ graph given that no divisor $d$ of $p\pm1$ satisfies $\frac{2\sqrt{2p}}{M_d}<d<\frac{81M_d^3}{4}$ or $\frac{p}{6M_d}< d < \frac{8\sqrt{p}(p\pm 1)\tau(p\pm 1)}{\phi(p\pm 1)}$. From Table \ref{tab:percentage}, we see that once we are on the order of $10^{21}$, Theorem \ref{thm:Md} captures almost all primes $p$. However there are still some exceptional cases where this theorem is inconclusive.

For the first 10,000 primes greater than $10^{21}$, there is a single prime $p'$ that does not pass these two criteria, $p'=1,000,000,000,000,000,124,399$. We have 
    \begin{align*}
        p'-1 &= 2 \cdot 7 \cdot 13 \cdot 29^2 \cdot 43 \cdot 705,737 \cdot 215,288,719 \\
        p'+1 &= 2^4 \cdot 3 \cdot 5^2 \cdot 11^2 \cdot 17 \cdot 19 \cdot 23 \cdot 97 \cdot 757 \cdot 1,453 \cdot 8,689\\
        \mathrm{Number\ of\ divisors\ of\ } p'\pm1 &= \tau(p-1)+\tau(p+1) -2 = 192+11,520 - 2= 11,710\\
        \mathrm{Number\ of\ divisors\ of\ } p'\pm1 & \mathrm{\ which\ fail\ either\ bound\ of\ Theorem~\ref{thm:Md}} = 989
    \end{align*}
    The largest value that $M_d=|\Mcal_d(p'-1)\cup\Mcal_d(p'+1)|$ attains as $d$ varies over the 989 divisors of $p'\pm 1$ that fail one of the bounds in Theorem~\ref{thm:Md} is 438. An example of a divisor $d$ with $M_d=438$ is $d=1,664,125,969$. For this divisor we have 
    \begin{align*}
        \frac{2\sqrt{2p'}}{438} \approx 2.042 \times 10^8 < d \approx 1.664 \times 10^9  < 1.702 \times 10^9 \approx \frac{81\cdot 438^3}{4}.
    \end{align*}        
        Note that $\frac{p'}{6M_d} \approx 3.80518\times 10^{17}$, $\frac{8\sqrt{p'}(p'+1)\tau(p'+1)}{\phi(p'+1)} \approx 1.427 \times 10^{16}$, and $\frac{8\sqrt{p'}(p'-1)\tau(p'-1)}{\phi(p'-1)} \approx 1.302 \times 10^{14}$ so there are no divisors that can ever satisfy the second bound of Theorem~\ref{thm:Md}. 

While examples like this become exceedingly rare, they persist throughout the range in which we are able to execute Theorem \ref{thm:Md}'s test. Indeed, we have verified that our test fails for every prime $p<10^{100}$ such that $p\pm 1$ is a reduced number as defined in \ref{def:redn}. There are 591 such primes, and there are certainly many others for which our test also fails, just not enough to be picked up by our random samples of 10,000.


\section{Appendix}\label{appendix}

In this section, we make more precise some of the implied constants in the proof of the following proposition in \cite{BGS}. The point of this is to determine exactly how large an order a triple must have in order to conclude that it is connected to $\mathcal C_p$ as in the End Game in \cite{BGS}.

\begin{proposition}[Explicit version of Proposition 7 in \cite{BGS}]\label{preciseBGSEG}For $d$ dividing $p-1$ or $p+1$, a Markoff triple of order $d$ belongs to $\Ccal_p$ provided \begin{equation}\label{eq:13}d > \frac{8\sqrt{p}(p\pm 1)\tau(p\pm 1)}{\phi(p\pm 1)}\end{equation} (where the $\pm$ is determined by whether $d$ divides $p-1$ or $p+1$).\end{proposition}

\begin{proof}Without loss of generality, let $d$ be the first coordinate order of some Markoff triple, and recall notation from (\ref{eq:7}). In Proposition 7 of \cite{BGS}, Bourgain, Gamburd, and Sarnak show that if $d$ is sufficiently large (at least $p^{1/2+\delta}$ for some $\delta > 0$ depending on $p$), then either the second or third coordinate in the orbit $$\left(r+r^{-1},\,\frac{(r+r^{-1})(r^{2n}s+r^{-2n}s^{-1})}{r-r^{-1}},\,\frac{(r+r^{-1})(r^{2n\pm1}s+r^{2n\pm1}s^{-1})}{r-r^{-1}}\right)$$ has order $p-1$ for some $n$. We will run through their argument and show that (\ref{eq:13}) is sufficient for the relevant inequalities to hold. Since every triple of order $p-1$ is in $\Ccal_p$ (Proposition 6 in \cite{BGS}), this will complete the proof.

First suppose $d\,|\,p-1$. We seek a solution $(x,y)\in\F_p^*$ to \begin{equation}\label{eq:14}\frac{(r+r^{-1})(sx+s^{-1}x^{-1})}{r-r^{-1}}=y+y^{-1}\end{equation} such that $x$ belongs to the cyclic subgroup of order $d$ (generated by $r$ in the notation above), and $y$ is a primitive root modulo $p$. We will show such a solution exists with a counting argument. 

Let $d' = (p-1)/d$, and given some $e$ dividing $p-1$, let $e' = (p-1)/e$. Consider the equation \begin{equation}\label{eq:12}\frac{(r+r^{-1})(sx^{d'}+s^{-1}x^{-d'})}{r-r^{-1}}=y^{e'}+y^{-e'}.\end{equation} Assume for the moment that $d'\geq e'$ so that the projective completion of the affine curve defined above is given by $$\frac{s(r+r^{-1})}{r-r^{-1}}X^{2d'}Y^{e'}+\frac{r+r^{-1}}{s(r-r^{-1})}Y^{e'}Z^{2d'}-X^{d'}Y^{2e'}Z^{d'-e'}-X^{d'}Z^{d'+e'}=0.$$ Call this curve $C$. Bourgain-Gamburd-Sarnak show that $C$ is irreducible over $\overline{\F}_p$. Furthermore, its geometric genus is bounded from above by $$\binom{\deg C-1}{2}-\sum_{P\in C}\binom{m_P}{2},$$ where $m_P$ denotes the multiplicity of the point $P$ in $C$. (See Corollary 1 in Section 8.3 of \cite{fulton}, for example.) Observe that $P=[0:1:0]$ has multiplicity $m_P=2d'-e'$, so the genus is at most $$\binom{2d'+e'-1}{2}-\binom{2d'-e'}{2} = 4d'e'-4d'-2e'+2.$$ Thus we can apply the Weil bound to conclude that the number of points on $C$ over $\F_p$ differs from $p+1$ by at most $2(4d'e'-4d'-2e'+2)\sqrt{p}$. Now let us exclude the points $[1:0:0]$, $[0:1:0]$, and $[0:0:1]$, which occur on $C$ with multiplicities $e'$, $2d'-e'$, and $e'$, respectively. Then, via the map $[X:Y:Z]\mapsto ((X/Z)^{d'},(Y/Z)^{e'})$, there is an $e'd'$-to-1 correspondence between the remaining points on $C$ and solutions to (\ref{eq:12}) in which $x$ belongs to the subgroup of order $d$ and $y$ to the subgroup of order $e$ in $\F_p^*$. In particular, if $f(e)$ denotes the number of such solutions $(x,y)$, then we have shown $$|d'e'f(e)+(e'+(2d'-e')+e')-(p+1)|<2(4d'e'-4d'-2e'+2)\sqrt{p}.$$ This simplifies to the following slightly weaker form: $$\left|f(e)-\frac{p+1}{d'e'}\right|<8\sqrt{p}.$$ The exact same bound can be obtained in the case $e'>d'$ by swapping $d'$ and $e'$ throughout the argument and using the singular point $[1:0:0]$ instead of $[0:1:0]$ to bound the genus.

Let $\mu$ be the M\"{o}bius function and let $\phi$ be Euler's totient function. By inclusion-exclusion, the number of solutions to (\ref{eq:12}) in which $x$ belongs to the cyclic group of order $d$ and $y$ is a primitive root is \begin{eqnarray*}\sum_{e\,|\,p-1}\mu\!\left(\frac{p-1}{e}\right)f(e)&\geq &\sum_{e'\,|\,p-1}\!\left(\mu(e')\frac{p+1}{d'e'}-8\sqrt{p}\right)\\&\geq&\frac{p+1}{d'}\\&=&\frac{(p+1)\phi(p-1)}{d'(p-1)}-8\sqrt{p}\,\tau(p-1)\\&>&\frac{d\phi(p-1)}{p-1}-8\sqrt{p}\,\tau(p-1).\end{eqnarray*} The last expression above is positive precisely when $d$ satisfies (\ref{eq:13}).

A very similar argument works when $d\,|\,p+1$. But now $r\not\in\F_p$, so a modification is needed in order to reapply the Weil bound over $\F_p$. Let $d'=(p+1)/d$. Instead of (\ref{eq:14}), we now count points on the curve $$\sum_{i=0}^{\lfloor d'/2\rfloor}\binom{d}{2i}x^{d'-2i}(1-x^2)^i=y^{e'}+y^{-e'},$$ where $e'$ is still some divisor of $p-1$ (see equation (42) in \cite{BGS}). The same singular points, $[0:1:0]$ when $d'\geq e'$ and $[1:0:0]$ when $e'\geq d'$, can be used to bound the genus of the curve above, and in fact we get an even smaller bound of $2d'e'$. The remainder of the proof is unchanged.\end{proof}

\printbibliography

\end{document}